\documentclass[11pt,a4paper]{article}
\usepackage[T1]{fontenc}
\usepackage[utf8]{inputenc}
\usepackage{lmodern,microtype}
\usepackage[margin=27mm]{geometry}
\usepackage{amsmath,amssymb,amsthm,mathtools,enumitem,booktabs}
\usepackage[dvipsnames]{xcolor}
\usepackage{hyperref}
\hypersetup{colorlinks=true,linkcolor=MidnightBlue,citecolor=MidnightBlue,urlcolor=MidnightBlue,
 pdftitle={Near-Horizon Scalar-Wave Sampling: Rindler Bounds and Schwarzschild Comparison},pdfauthor={Anik Chakraborty}}
\setlist[enumerate]{label=(\roman*),leftmargin=*,itemsep=3pt}
\numberwithin{equation}{section}
\newtheorem{theorem}{Theorem}[section]
\newtheorem{proposition}[theorem]{Proposition}
\newtheorem{lemma}[theorem]{Lemma}
\newtheorem{corollary}[theorem]{Corollary}
\theoremstyle{definition}
\newtheorem{definition}[theorem]{Definition}
\theoremstyle{remark}
\newtheorem{remark}[theorem]{Remark}
\newcommand{\R}{\mathbb R}
\newcommand{\C}{\mathbb C}
\newcommand{\N}{\mathbb N}
\newcommand{\Hh}{\mathcal H}
\newcommand{\Oo}{\mathcal O}
\newcommand{\Aa}{\mathcal A}
\newcommand{\Ee}{\mathcal E}
\newcommand{\Dom}{\mathcal D}
\newcommand{\ind}{\mathbf 1}

\newcommand{\ip}[2]{\left\langle #1,#2\right\rangle}
\DeclareMathOperator{\supp}{supp}
\DeclareMathOperator{\ran}{ran}
\DeclareMathOperator{\cond}{cond}
\DeclareMathOperator{\essinf}{ess\,inf}
\DeclareMathOperator{\esssup}{ess\,sup}

\title{Near-Horizon Scalar-Wave Sampling:\\Rindler Bounds and Schwarzschild Comparison}
\author{Anik Chakraborty\\[3pt]
\small Department of Mathematics, University of Delhi, India\\
\small\texttt{achakraborty@maths.du.ac.in}}
\date{\today}
\begin{document}
\maketitle
\begin{abstract}
We study classical reconstruction of an outgoing scalar-wave component from radial-shell observations. Starting from the noncompact-time Rindler metric, we construct Klein--Gordon evolution by closed quadratic forms and reduce the massless characteristic dynamics to translation in optical distance. For an initial optical interval of length $L$, a shell of width $d$, and recording time $T$, the optimal squared observability bounds are $A(T,L)=\min\{d,(T-L)_+\}$ and $B(T,L)=\min\{d,T\}$. At $T=L$, the Rindler observation map is injective with nonclosed range. Discrete shell observations form a spatial fusion frame precisely when their transported regions cover the initial interval; we compute optimal uniform-grid bounds and characterize robustness under profile erasures. Finite-resolution reconstruction separates localization error, approximation error and noise amplification. For fixed Schwarzschild mass and angular momentum, we compare the scalar-wave observation operator with its Rindler limit. Anchored finite-energy preparation permits arbitrary prescribed outgoing $L^2([-L,0])$ profiles without a zero-integral condition. The operator error is explicitly $O(e^{-2\kappa a})$ as an apparatus of fixed optical size is translated toward the horizon; whenever $T>L$, positive Rindler bounds transfer to Schwarzschild for all sufficiently large $a$. A positive Schwarzschild lower bound is impossible at or below the optical travel time, for every translation. These results concern classical observation and reconstruction, not quantum information or dissipative horizon evolution.
\end{abstract}

\section{Introduction}\label{sec:intro}
Reconstructing a wave from measurements requires a choice of state space, detector and recording interval. Even when the underlying evolution conserves energy, a restricted observation can have a kernel or an unbounded inverse. Near a horizon, the distinction is particularly useful: a long optical travel time can obstruct observation by a prescribed apparatus without implying any failure of the wave evolution itself.

We consider a massless scalar field on a Rindler wedge and measure its outgoing characteristic component on a fixed radial shell. The unknown initial component is supported below that shell. In the optical coordinate, the wave travels at unit speed, and the observation Gramian records the time each characteristic spends in the detector. This gives exact stability constants and a sharp recording-time threshold. The same calculation has a discrete version in which transported observation regions form a fusion frame. Its redundancy has a concrete meaning: repeated coverage of an initial spatial point protects that part of the state against erased observations.

The relation between operator-generated frames and dynamical sampling was developed in \cite{ACMT,ACMPT}; bounded operator representations and their duality and stability properties are studied in \cite{CH,CHP}. The observability interpretation is established explicitly by D\'iaz Mart\'in, Medri and Molter \cite{DMMcontrol}, whose continuous-time work also addresses discretization \cite{DMMcontinuous}. Our finite-channel obstructions are auxiliary consequences of that established viewpoint. The orbit and fusion-orbit perspective in Osgooei and Abdollahpour \cite{OA} is useful motivation, but the contraction estimate and all reconstruction identities used here are proved directly. Frames of subspaces and their dual synthesis maps are used in the sense of \cite{CK,HMBZ}.

The transport mechanism itself is also established. Vaidya, Rajaram and Dasgupta obtain multiplication-operator observability Gramians for advection equations \cite{VRD}. The continuous Gramian below is a one-dimensional specialization of this mechanism, followed by evaluation of its essential extrema in Rindler optical coordinates. More general geometric conditions for wave observation are provided by Bardos, Lebeau and Rauch \cite{BLR} and, for time-dependent observation domains, by Le Rousseau, Lebeau, Terpolilli and Tr\'elat \cite{LRLT}. The frame--observability correspondence and the underlying geometric-control principles are established tools in this setting.

Quantitative degeneration near the minimal observation time and dependence on lower-order potentials are studied by Laurent and L\'eautaud \cite{LL}. The critical-time packet argument below is a direct instance of this established obstruction mechanism for our prepared characteristic data.

The principal additional step is a quantitative comparison with the actual Schwarzschild scalar equation at fixed angular momentum. Both the initial region and the detector are translated toward the horizon, with their optical sizes and the recording time fixed. We use finite-energy data whose displacement is anchored to zero on the outer side. This permits every outgoing $L^2$ profile on the prescribed interval; its integral need not vanish. An exponential bound on the scalar potential gives an explicit observation-operator error and stable reconstruction for sufficiently near-horizon apparatus. We also prove that a positive lower bound is impossible at or below the optical travel time, for every translation. These statements are local in the apparatus and finite in time; they are distinct from global black-hole wave decay estimates such as those discussed in \cite{DR}.

Free characteristic comparison at the Schwarzschild horizon and finite-energy scattering constructions are already present in \cite{Dimock,Nicolas}. In particular, the Cook argument in \cite[Appendix A]{Nicolas} uses the exponential decay of the horizon potential. Our comparison concerns an explicit finite-time observation-operator norm on a fixed interval of anchored derivative profiles and the resulting reconstruction bounds.

The geometry is introduced through the field equation, rather than through an imposed horizon-vanishing multiplier. The classical Rindler setup underlying the field analysis is familiar from \cite{Fulling}, and self-adjoint spatial operators provide a systematic construction of static scalar dynamics \cite{IW}. Here the propagators are conservative. A separate time-dependent detector response is considered in Section~\ref{sec:gain}; its attenuation is not a property of the propagator.

The paper provides exact continuous and discrete observation bounds, sharp whole-profile erasure criteria, and regularized reconstruction estimates in Rindler optical coordinates. The main comparison, Lemma~\ref{lem:energycompletion} and Theorem~\ref{thm:Schwarzcomparison}, treats every prescribed outgoing profile on a fixed interval through anchored finite-energy data. It gives an explicit finite-time observation-operator estimate and transferred reconstruction bounds for a fixed angular mode and a translated apparatus of fixed optical size. The specific contribution lies in this quantitative observation formulation of the established comparison mechanisms.

\paragraph{Conventions.}
All Hilbert spaces are complex and separable; inner products are linear in the first variable. Functions on a specified initial interval are extended by zero unless stated otherwise. Endpoints of intervals do not affect $L^2$ identities. Bounded operators have their entire stated spaces as domains and are closed. All reconstruction limits are in Hilbert-space norm, unless operator-norm convergence is specified. We write $x_+=\max\{x,0\}$, and use units in which the propagation speed is one.

\section{Rindler geometry and scalar-wave evolution}\label{sec:rindler}
\subsection{Metric, optical distance and the spatial operator}
Fix $\kappa,\rho_0>0$ and use the metric
\begin{equation}\label{eq:metric}
g=-(\kappa\rho)^2\,d\tau^2+d\rho^2,
\qquad (\tau,\rho)\in\R\times(0,\infty).
\end{equation}
The coordinate $\tau$ is not periodic. With
\begin{equation}\label{eq:optical}
s=\kappa^{-1}\log(\rho/\rho_0),\qquad
ds=\frac{d\rho}{\kappa\rho},
\end{equation}
the horizon corresponds to $s=-\infty$ and the metric is
$(\kappa\rho_0)^2e^{2\kappa s}(-d\tau^2+ds^2)$.

\begin{proposition}[Klein--Gordon realization]\label{prop:KG}
For $m\ge0$, the equation $(\Box_g-m^2)u=0$ is
\begin{equation}\label{eq:rhoKG}
u_{\tau\tau}+\mathcal L_{\rho,m}u=0,
\qquad
\mathcal L_{\rho,m}=-\kappa^2\rho\partial_\rho(\rho\partial_\rho)
+m^2\kappa^2\rho^2.
\end{equation}
The coordinate map is unitary from
$\Hh_\rho=L^2((0,\infty),d\rho/(\kappa\rho))$ onto $L^2(\R,ds)$.
In the latter space define
\begin{align}
V_m(s)&=m^2\kappa^2\rho_0^2e^{2\kappa s},\label{eq:masspotential}\\
a_m[u,v]&=\int_\R \big(u'\overline{v'}+V_m u\overline v\big)\,ds,\qquad
\Dom(a_m)=\{u\in H^1(\R):\sqrt{V_m}u\in L^2\}.
\end{align}
This is a densely defined, closed, nonnegative form. Its associated operator is the nonnegative self-adjoint Friedrichs realization
\begin{equation}\label{eq:Ldomain}
L_m=-\partial_s^2+V_m,\qquad
\Dom(L_m)=\{u\in\Dom(a_m):-u''+V_m u\in L^2\text{ in distributions}\}.
\end{equation}
It is closed and unbounded. Its square root $H_m=L_m^{1/2}$ is self-adjoint with
$\Dom(H_m)=\Dom(a_m)$, and $e^{-itH_m}$ is a strongly continuous unitary group on $L^2(\R)$.
\end{proposition}
\begin{proof}
Since $\sqrt{|\det g|}=\kappa\rho$, the divergence formula gives
\[
\Box_g u=-\frac{u_{\tau\tau}}{\kappa^2\rho^2}
+\frac1\rho\partial_\rho(\rho\partial_\rho u).
\]
Multiplication by $\kappa^2\rho^2$ proves \eqref{eq:rhoKG}.
Equation~\eqref{eq:optical} proves unitarity of the coordinate change, and
$\partial_s=\kappa\rho\partial_\rho$ gives \eqref{eq:masspotential}.

The form contains $C_c^\infty(\R)$. A Cauchy sequence for
$\|u\|_2^2+a_m[u,u]$ converges in $H^1$ and its products with $\sqrt{V_m}$ converge in $L^2$. Almost-everywhere subsequences identify the latter limit with the product of the former limit. Thus the form is closed. Cutoffs followed by smoothing show that $C_c^\infty$ is a form core: the potential tail is integrable for a form-domain element, and the potential is bounded on every fixed compact set. The representation theorem for closed nonnegative forms gives a nonnegative self-adjoint operator. Testing against $C_c^\infty$ identifies its distributional action; form-core density then gives precisely \eqref{eq:Ldomain}.

For a fixed nonzero compactly supported smooth $\chi$, the form values of $e^{ins}\chi$ grow quadratically in $n$. Consequently $H_m$, and hence $L_m$, is unbounded. The spectral theorem defines the square root and the unitary group, with adjoint $e^{itH_m}$; Stone's theorem identifies its generator as $-iH_m$. The radial realization is the unitary pullback of \eqref{eq:Ldomain}. No finite endpoint or horizon boundary condition is introduced.
\end{proof}

The measure in $\Hh_\rho$ is the spatial measure divided by the lapse. It is not the spacetime volume measure. Changing it to Lebesgue radial measure by a unitary field rescaling would also change the represented operator; this does not introduce attenuation of the invariant norm.

\begin{proposition}[Wave dynamics and energy]\label{prop:waveenergy}
For $u_0\in\Dom(H_m)$ and $v_0\in L^2(\R)$, spectral calculus defines the weak wave solution
\begin{equation}\label{eq:wavecalc}
u(t)=\cos(tH_m)u_0+b_t(H_m)v_0,\qquad
b_t(\lambda)=\begin{cases}\sin(t\lambda)/\lambda,&\lambda>0,\\t,&\lambda=0.\end{cases}
\end{equation}
It belongs to $C(\R;\Dom(H_m))\cap C^1(\R;L^2)$, with the form graph norm on the first space, and solves the equation in the form-domain dual. Its conserved energy is
\begin{align}
E_m[u](t)&=\frac12\big(\|u_t(t)\|_2^2+\|H_mu(t)\|_2^2\big)\label{eq:energy}\\
&=\frac12\int_0^\infty
\left(\frac{|u_\tau|^2}{\kappa\rho}
+\kappa\rho|u_\rho|^2+m^2\kappa\rho|u|^2\right)d\rho.
\end{align}
The wave evolution extends to a strongly continuous unitary group on the completion of $\Dom(H_m)\times L^2$ for the norm $(2E_m)^{1/2}$.
\end{proposition}
\begin{proof}
The estimates $|b_t(\lambda)|\le |t|$ and $|\lambda b_t(\lambda)|\le1$ give the stated regularity by dominated convergence in the spectral representation. In the variables $(p,v)=(H_mu,u_t)$, evolution is
\begin{equation}\label{eq:rotation}
\begin{pmatrix}p(t)\\v(t)\end{pmatrix}
=\begin{pmatrix}\cos(tH_m)&\sin(tH_m)\\-\sin(tH_m)&\cos(tH_m)\end{pmatrix}
\begin{pmatrix}H_mu_0\\v_0\end{pmatrix}.
\end{equation}
The matrix is unitary by spectral calculus, proving energy conservation and the completion statement. Its closed skew-adjoint generator is
$\left(\begin{smallmatrix}0&H_m\\-H_m&0\end{smallmatrix}\right)$ on
$\Dom(H_m)\oplus\Dom(H_m)$. Indeed, the corresponding constant matrix diagonalization reduces it to $iH_m$ and $-iH_m$. There is no $L^2$ kernel of $H_m$: zero form energy forces a constant function, which cannot be a nonzero element of $L^2(\R)$. Thus $\ran H_m$ is dense and the completed coordinates in \eqref{eq:rotation} fill $L^2\oplus L^2$. The radial energy identity follows by substitution in \eqref{eq:optical}. Spectral truncation proves the weak equation and uniqueness first on a spectral core and then by the displayed continuous energy evolution.
\end{proof}

\subsection{Massless characteristic data}
\begin{proposition}[Conservative characteristic evolution]\label{prop:characteristics}
For $m=0$, let $\dot H^1(\R)$ denote the completion of $C_c^\infty(\R)$ in the derivative norm, represented by locally $H^1$ functions modulo constants. On the homogeneous wave-energy space $\dot H^1\times L^2$, set
\begin{equation}\label{eq:qdef}
q_+=\frac{u_t-u_s}{2},\qquad q_-=\frac{u_t+u_s}{2}.
\end{equation}
Then
\begin{equation}\label{eq:qtranslation}
q_+(t,s)=q_+(0,s-t),\qquad q_-(t,s)=q_-(0,s+t),
\qquad E_0=\|q_+\|_2^2+\|q_-\|_2^2.
\end{equation}
The outgoing group $V(t)f(s)=f(s-t)$ equals $e^{-itD}$, where
$D=-i\partial_s$ is self-adjoint on $H^1(\R)\subset L^2(\R)$.
\end{proposition}
\begin{proof}
The derivative map identifies $\dot H^1$ with $L^2$. Its range is dense because an $L^2$ function orthogonal to all derivatives of test functions has distributional derivative zero and is therefore an $L^2$ constant, hence zero. Primitives represent the completion modulo constants. The identity for energy follows from the parallelogram law. For smooth solutions, $u_{tt}=u_{ss}$ gives
$(\partial_t+\partial_s)q_+=0$ and $(\partial_t-\partial_s)q_-=0$.
Characteristics give \eqref{eq:qtranslation}; density extends it to the energy space. Fourier transformation makes $D$ multiplication by the real variable $\xi$, with domain $\{f:\xi\widehat f\in L^2\}=H^1$. Its spectral exponential is the asserted translation.
\end{proof}

Here and below $t=\tau$. The positive-frequency group generated by $H_0=|D|$ and the outgoing group generated by $D$ are different. The latter propagates one characteristic energy component. Neither the full wave evolution nor this outgoing group is strictly contractive.

\section{Observation operators and general obstructions}\label{sec:framework}
\subsection{Analysis operators, frames and reconstruction}
\begin{definition}\label{def:observation}
Let $U(t)$ be a strongly continuous unitary group on a Hilbert space $\Hh$, let $J:X\to\Hh$ be an isometric embedding, and let $C:\Hh\to Y$ be bounded. For $0<T<\infty$ the observation map is
\[
\Oo f(t)=CU(t)Jf\in L^2(0,T;Y).
\]
It is bounded with $\|\Oo\|\le\sqrt T\|C\|$. Its Gramian is
$S=\Oo^*\Oo=\int_0^T J^*U(t)^*C^*CU(t)J\,dt$, where the integral is understood weakly. A positive lower observation bound is a constant $A>0$ such that
$A\|f\|^2\le\|\Oo f\|^2$ for all $f\in X$.
\end{definition}

With an orthonormal basis $(e_j)$ of $Y$, put
$\phi_{j,t}=J^*U(t)^*C^*e_j$. Parseval and Tonelli give
\begin{equation}\label{eq:frameobserv}
\|\Oo f\|^2=\int_0^T\sum_j|\ip{f}{\phi_{j,t}}|^2dt.
\end{equation}
Thus upper and lower observation bounds are exactly the Bessel and frame inequalities for this semi-continuous family. This is the standard correspondence of \cite{DMMcontrol,DMMcontinuous}. No point-evaluation functional on $L^2$ is used in \eqref{eq:frameobserv}.

\begin{lemma}[Dual reconstruction and perturbation]\label{lem:inverse}
Let $\Aa:X\to Y$ be bounded, with optimal squared bounds $0<A\le B<\infty$. Then its range is closed and
\begin{equation}\label{eq:canonical}
R=(\Aa^*\Aa)^{-1}\Aa^*,\qquad
R\Aa=I,\qquad \|R\|=A^{-1/2},\qquad
\cond(\Aa):=\|\Aa\|\|R\|=\sqrt{B/A}.
\end{equation}
All bounded left inverses are
\begin{equation}\label{eq:allleft}
R+K(I-\Aa R),\qquad K\in\mathcal B(Y,X).
\end{equation}
If $\widetilde\Aa:X\to Y$ satisfies $\|\widetilde\Aa-\Aa\|\le h<\sqrt A$, then valid bounds for it are
\begin{equation}\label{eq:perturb}
(\sqrt A-h)^2\|f\|^2\le\|\widetilde\Aa f\|^2
\le(\sqrt B+h)^2\|f\|^2.
\end{equation}
Its canonical inverse is $(\widetilde\Aa^*\widetilde\Aa)^{-1}\widetilde\Aa^*$.
\end{lemma}
\begin{proof}
The lower bound implies that a Cauchy sequence of images has Cauchy preimages, proving closed range. With $S=\Aa^*\Aa$, $AI\le S\le BI$, so $S^{-1}$ exists and $R\Aa=I$. Since $RR^*=S^{-1}$, its norm is $A^{-1/2}$. Moreover $\Aa R$ is the orthogonal projection onto $\ran\Aa$. Every map in \eqref{eq:allleft} is a left inverse; conversely, a left inverse $L$ satisfies $L=R+L(I-\Aa R)$. Triangle and reverse triangle inequalities prove \eqref{eq:perturb}.
\end{proof}

If the observation map in \eqref{eq:frameobserv} has a positive lower bound, its Gramian $S$ is boundedly invertible and the canonical dual vectors are $S^{-1}\phi_{j,t}$. The reconstruction
\[
f=\sum_j\int_0^T\ip{f}{\phi_{j,t}}S^{-1}\phi_{j,t}\,dt
\]
means the norm limit of the bounded synthesis map on finite spatial truncations. These Bochner integrals and their limit are justified by the Bessel bound. Exact reconstruction is an identity for noiseless data; stable reconstruction also requires boundedness of its inverse. A regularized reconstruction will instead involve an explicitly stated cutoff or approximation space.

\subsection{What finitely many bounded channels cannot observe}
The next statements are standard compactness and tail arguments. They delimit the role of operator orbits; they do not describe attenuation by the Rindler metric. Related frame obstructions occur in \cite{ACMPT,DMMcontrol,DMMcontinuous}. The strict-contraction upper estimate also appears in \cite[Proposition 2.1]{OA}.

\begin{proposition}[Finite-channel obstructions]\label{prop:obstructions}
Let $\Hh$ be infinite dimensional.
\begin{enumerate}
\item If $T\in\mathcal B(\Hh)$, $\|T\|\le q<1$, and $\phi_1,\ldots,\phi_J\in\Hh$, then
$\Aa f=(\ip{f}{T^k\phi_j})_{j,k\ge0}$ is Hilbert--Schmidt and
\begin{equation}\label{eq:HS}
\|\Aa\|_{\rm HS}^2=\sum_{j,k}\|T^k\phi_j\|^2
\le\frac{\sum_j\|\phi_j\|^2}{1-q^2}=:M_q.
\end{equation}
It has no positive lower frame bound. On an $r$-dimensional subspace, any lower bound $A_r$ satisfies $rA_r\le M_q$. The same conclusions hold for countably many windows with summable squared norms.
\item If $U(t)$ is a strongly continuous unitary group and $C:\Hh\to\C^J$ is bounded, then $f\mapsto CU(\cdot)f$ into $L^2(0,T;\C^J)$ is Hilbert--Schmidt for finite $T$, with squared Hilbert--Schmidt norm $T\|C\|_{\rm HS}^2$. Finitely many discrete observation times give a finite-rank map. Neither map is bounded below.
\item If $U$ is unitary and $C:\Hh\to Y$ is bounded, a one-sided observation satisfying $\sum_{k\ge0}\|CU^kf\|^2\le B\|f\|^2$ cannot have a positive lower bound on all of $\Hh$.
\end{enumerate}
\end{proposition}
\begin{proof}
For an orthonormal basis $(e_n)$, Tonelli and Parseval give
$\sum_n\|\Aa e_n\|^2=\sum_{j,k}\|T^k\phi_j\|^2$; the geometric series gives \eqref{eq:HS}. Summing a putative lower bound over $r$ orthonormal vectors gives $rA_r\le M_q$, which rules out a fixed positive bound in infinite dimension. For (ii), write $Cf=(\ip{f}{\phi_j})_j$ and integrate $\sum_j\|U(t)^*\phi_j\|^2=\sum_j\|\phi_j\|^2$. Finite discrete data have finite-dimensional codomain. For (iii), fix $f\ne0$. The squared observation norm of $U^nf$ is $\sum_{k=n}^\infty\|CU^kf\|^2\to0$, although $\|U^nf\|=\|f\|$.
\end{proof}

\begin{remark}\label{rem:obstructions}
Strict contractivity is stronger than strong stability $T^nf\to0$. One-sided unitary orbits can be Bessel; part (iii) excludes a global lower bound, not Besselness. The proposition does not cover unbounded instantaneous observations with bounded time-integrated realizations. For transport, for example, the trace $q(t,0)=f(-t)$ is an isometry from $L^2(-T,0)$ to $L^2(0,T)$, although point evaluation on $L^2(\R)$ is not bounded. The shell detector used below has infinite spatial rank, and the initial support subspace is not invariant under translation. These facts explain why the positive results do not contradict the obstructions.
\end{remark}

\section{Exact Rindler observation bounds}\label{sec:exact}
Choose $0<\varepsilon<\rho_0<\rho_1$ and put
\begin{equation}\label{eq:Ld}
L=\kappa^{-1}\log(\rho_0/\varepsilon),\qquad
d=\kappa^{-1}\log(\rho_1/\rho_0),\qquad
X_L=L^2([-L,0]),\quad I=[0,d].
\end{equation}
The unknown is $f=q_+(0,\cdot)\in X_L$, and initially $q_-=0$. The measured quantity is the characteristic amplitude
$q_+=(\partial_tu-\kappa\rho\partial_\rho u)/2$ on the shell $\rho_0\le\rho\le\rho_1$, recorded in the norm $L^2(dt\,d\rho/(\kappa\rho))$. This idealized distributed observation requires a spatial profile, not finitely many bounded scalar readings.

Let $J_L:X_L\to L^2(\R)$ be zero extension and let $C_I$ be restriction to $I$. The observation is the bounded map
\begin{equation}\label{eq:Odef}
\Oo_{T,L}:X_L\longrightarrow Y_T:=L^2((0,T)\times I),\qquad
(\Oo_{T,L}f)(t,s)=f(s-t)=C_IV(t)J_Lf(s).
\end{equation}
All initial and observed intervals in this paper specify supports or detectors, not reflecting wave boundaries.

\begin{theorem}[Sharp continuous observation bounds]\label{thm:exact}
For $L,d,T>0$, the Gramian of \eqref{eq:Odef} is
\begin{equation}\label{eq:weight}
\Oo_{T,L}^*\Oo_{T,L}=M_{w_T},\qquad
w_T(z)=|[0,T]\cap[-z,d-z]|=\min\{d,(T+z)_+\},\quad -L\le z\le0.
\end{equation}
The optimal squared lower and upper bounds are
\begin{equation}\label{eq:AB}
A(T,L)=\min\{d,(T-L)_+\},\qquad B(T,L)=\min\{d,T\}.
\end{equation}
More precisely:
\begin{enumerate}
\item If $T<L$, then $\ker\Oo_{T,L}=L^2((-L,-T))$, embedded in $X_L$.
\item If $T=L$, then $\Oo_{T,L}$ is injective, is not bounded below, and has nonclosed range.
\item If $T>L$, then its canonical bounded left inverse is
\begin{equation}\label{eq:Rexplicit}
(R_{T,L}y)(z)=\frac1{w_T(z)}
\int_{\max(0,-z)}^{\min(T,d-z)} y(t,z+t)\,dt,
\end{equation}
with $\|R_{T,L}\|=A(T,L)^{-1/2}$ and
\begin{equation}\label{eq:condexact}
\cond(\Oo_{T,L})=
\sqrt{\frac{\min(d,T)}{\min(d,T-L)}}.
\end{equation}
For $T\ge L+d$ the observation is tight, with both squared bounds equal to $d$.
\end{enumerate}
\end{theorem}
\begin{proof}
Tonelli's theorem and the change of variables $z=s-t$ give
\[
\|\Oo_{T,L}f\|^2
=\int_{-L}^0 |f(z)|^2\int_0^T\ind_{[0,d]}(z+t)\,dt\,dz.
\]
For $z\le0$, the inner integral is \eqref{eq:weight}. The Gramian follows by polarization. The multiplier is nondecreasing in $z$; its essential infimum and supremum are exactly \eqref{eq:AB}. Normalized functions supported in successively smaller endpoint neighborhoods show that these constants are optimal.

The zero set of the multiplier gives (i). At $T=L$, it is positive almost everywhere, so the map is injective. A unit vector supported in $[-L,-L+\eta]$ has squared observation norm at most $\eta$. A bounded injective operator with closed range is bounded below, by the bounded inverse theorem on its range. This proves (ii).

For test data, a change of variables in $\ip{\Oo f}{y}$ shows that the adjoint is the characteristic integral in \eqref{eq:Rexplicit}, without the factor $1/w_T$. Cauchy--Schwarz along each characteristic shows that this defines a bounded map for general $y\in Y_T$, with the formula valid for almost every $z$. If $T>L$, multiplication by $1/w_T$ is bounded and gives the left inverse. Lemma~\ref{lem:inverse} proves its exact norm and the condition number. If $T\ge L+d$, every characteristic is recorded for the full shell-crossing time $d$.
\end{proof}

\begin{corollary}[Radial interpretation]\label{cor:radial}
For this detector and data norm, stable reconstruction of every $f\in X_L$ during $[0,T]$ is equivalent to
\begin{equation}\label{eq:radialthreshold}
T>\kappa^{-1}\log(\rho_0/\varepsilon)
\quad\Longleftrightarrow\quad
\varepsilon>\rho_0e^{-\kappa T}.
\end{equation}
If $T=L+\sigma$ with fixed $\sigma>0$, its lower bound remains $\min(d,\sigma)$ as $L$ increases.
\end{corollary}
\begin{proof}
Substitute \eqref{eq:Ld} in \eqref{eq:AB}. Equivalently, an outgoing radial null curve satisfies $d\rho/dt=\kappa\rho$, so it takes time $L$ to travel from $\varepsilon$ to $\rho_0$.
\end{proof}

The threshold expresses insufficient recording time for the chosen detector. At fixed $T$ the lower bound vanishes at a positive cutoff radius, before a horizon limit is reached. Allowing additional recording time prevents this degeneration. An initially incoming component below the shell travels to smaller $s$ and is not observed by this outer detector in positive time. The theorem recovers only the outgoing data specified above.

\section{Discrete sampling and spatial fusion frames}\label{sec:discrete}
\subsection{Coverage and optimal uniform-grid bounds}
For $\Delta>0$ and $N\ge1$, record complete shell profiles at $t_k=k\Delta$, including $k=N$. Use the normalization
\begin{equation}\label{eq:discreteanalysis}
\Aa_{\Delta,N}:X_L\longrightarrow\bigoplus_{k=0}^N L^2(I),\qquad
(\Aa_{\Delta,N}f)_k(s)=\sqrt\Delta\,f(s-k\Delta).
\end{equation}
The subspaces carrying the transported measurements are
\begin{equation}\label{eq:Wk}
E_k=[-L,0]\cap[-k\Delta,d-k\Delta],\qquad
W_k=\{f\in X_L:\supp f\subset E_k\}.
\end{equation}
They are closed, and $P_{W_k}=M_{\ind_{E_k}}$. A weighted fusion frame for $X_L$ is a family of closed subspaces and positive weights $(W_k,v_k)$ for which there exist constants $0<A\le B<\infty$ such that
$A\|f\|^2\le\sum_kv_k^2\|P_{W_k}f\|^2\le B\|f\|^2$ for every $f\in X_L$.
We omit zero subspaces when using this terminology.

\begin{theorem}[Discrete radial fusion frame]\label{thm:discrete}
The map \eqref{eq:discreteanalysis} has Gramian
\begin{equation}\label{eq:discreteweight}
\Aa_{\Delta,N}^*\Aa_{\Delta,N}=M_{w_{\Delta,N}},\qquad
w_{\Delta,N}(z)=\Delta\sum_{k=0}^N\ind_I(z+k\Delta).
\end{equation}
It is bounded below, equivalently $(W_k,\sqrt\Delta)$ is a fusion frame for $X_L$, if and only if
\begin{equation}\label{eq:coveragecondition}
d\ge\Delta,\qquad N\Delta\ge L.
\end{equation}
Under these conditions, write
\[
a=L/\Delta,\qquad b=d/\Delta=p+r,\qquad
p=\lfloor b\rfloor,\quad 0\le r<1,
\]
and set $\chi=1$ when $r>0$ and $a>1-r$, and $\chi=0$ otherwise.
The optimal squared bounds are
\begin{equation}\label{eq:optimaldiscrete}
A_{\Delta,N}=\Delta\min\{p,N-\lceil a\rceil+1\},\qquad
B_{\Delta,N}=\Delta\min\{N,p+\chi\}.
\end{equation}
The canonical left inverse is
\begin{equation}\label{eq:discreteinverse}
(R_{\Delta,N}y)(z)=\frac{\sqrt\Delta}{w_{\Delta,N}(z)}
\sum_{k=0}^N\ind_I(z+k\Delta)y_k(z+k\Delta).
\end{equation}
Its norm is $A_{\Delta,N}^{-1/2}$ and its relative condition number is $\sqrt{B_{\Delta,N}/A_{\Delta,N}}$.
\end{theorem}
\begin{proof}
Changing variables separately at each observation time gives
$\|\Aa_{\Delta,N}f\|^2=\int_{-L}^0w_{\Delta,N}|f|^2$.
This proves the Gramian identity and the fusion-frame equivalence. If $N\Delta<L$, a left subinterval of $[-L,0]$ is uncovered. If $d<\Delta$, a positive-length interval immediately to the left of zero is uncovered: $k=0$ sees no negative initial point, and every $k\ge1$ interval ends at or below $d-\Delta<0$. Conversely, if \eqref{eq:coveragecondition} holds, the intervals $[-k\Delta,d-k\Delta]$ meet or overlap and cover $[-N\Delta,d]$. Their coverage of $[-L,0]$ is at least one almost everywhere.

To compute the essential multiplicities, put $x=-z/\Delta\in(0,a)$. For almost every $x\in(0,a)$, excluding the finitely many points at which $x\in\mathbb Z$ or $x+b\in\mathbb Z$, the count is
\begin{equation}\label{eq:countformula}
n(x)=\#\{k\in\{1,\ldots,N\}:x\le k\le x+b\}
=\min\{p+\ind_{\{\{x\}>1-r\}},\,N-\lceil x\rceil+1\},
\end{equation}
where the indicator is zero when $r=0$ and $\{x\}$ is the fractional part. Since $a\le N$, the second entry is positive. Its smallest value is $N-\lceil a\rceil+1$, attained near $a$ from the left. The first entry has minimum $p$, attained near zero. Thus the essential minimum is the minimum in \eqref{eq:optimaldiscrete}. For the maximum, the additional count $p+1$ occurs somewhere in $(0,a)$ precisely when $r>0$ and $a>1-r$; if so it already occurs in the first unit cell, before the terminal cap can reduce the count below $\min(N,p+1)$. Otherwise the maximum is $\min(N,p)$. This proves the upper formula. Each stated count holds on an interval of positive length, proving optimality. Applying the adjoint and dividing by the multiplier gives \eqref{eq:discreteinverse}; Lemma~\ref{lem:inverse} finishes the proof.
\end{proof}

In radial variables, $E_k$ is the initial interval intersected with
$[\rho_0e^{-\kappa k\Delta},\rho_1e^{-\kappa k\Delta}]$. The fusion-frame condition therefore says that transported radial observation regions cover the whole initial region. When $d=\Delta$ and $N\Delta\ge L$, the frame is tight with squared bound $\Delta$.

\begin{remark}[Endpoint sampling and dual blocks]\label{rem:discretedual}
Assume the coverage condition \eqref{eq:coveragecondition} throughout this remark. At $N\Delta=L$, the discrete lower bound is $\Delta$, although the continuous lower bound at $T=L$ is zero. The last full spatial profile carries positive mass in the discrete norm; one time has zero mass in the continuous norm. There is no contradiction, and the discrete bound tends to zero as this critical-time grid is refined.

For the fusion analysis $\mathcal C f=(\sqrt\Delta P_{W_k}f)_k$ and synthesis
$\mathcal T(g_k)=\sqrt\Delta\sum_k g_k$, the block map
$Q(g_k)=(M_{1/w_{\Delta,N}}g_k)_k$ preserves each $W_k$ and satisfies
$\mathcal TQ\mathcal C=I$. Indeed, multiplication by $1/w_{\Delta,N}$ is bounded and commutes with the support projections. This supplies a component-preserving dual reconstruction in the sense of \cite{HMBZ}, with the reconstruction identity explicitly verified. Mere preservation of the blocks would not imply that identity.
\end{remark}

\subsection{Nonuniform observations and erased profiles}
The same geometry gives an exact robustness statement without assumptions on small operator-norm perturbations of interval indicators.

General sufficient bounds after erasure of fusion-frame subspaces are given in \cite[Theorem 3.2]{CKerasures}. The commuting support projections here make it possible to compute the optimal common surviving bound and the exact uniform-grid threshold.

\begin{theorem}[Weighted coverage and erasures]\label{thm:erasures}
Let there be finitely many observations indexed by $i=1,\ldots,J$, with times $t_i\ge0$, bounded shell intervals $I_i\subset\R$, and weights $v_i>0$. Define
\[
(\mathcal Bf)_i(s)=v_i f(s-t_i),\quad s\in I_i,
\qquad E_i=[-L,0]\cap(I_i-t_i),\qquad b_i(z)=v_i^2\ind_{E_i}(z).
\]
The Gramian is multiplication by $w=\sum_i b_i$, with optimal bounds $\essinf w$ and $\esssup w$ on $[-L,0]$. The subspaces $L^2(E_i)$ with weights $v_i$ form a fusion frame exactly when $\essinf w>0$.

Fix $0\le r\le J$. Order $b_i(z)$ in decreasing order as $b_{(1)}(z),\ldots,b_{(J)}(z)$, and define
\begin{equation}\label{eq:robustbound}
A^{(r)}=\essinf_{z\in[-L,0]}
\left(w(z)-\sum_{j=1}^r b_{(j)}(z)\right).
\end{equation}
This is the optimal common lower bound after arbitrary deletion of at most $r$ whole observation profiles. Every such deletion permits stable reconstruction if and only if $A^{(r)}>0$. If all weights equal $v$, this is equivalent to coverage multiplicity at least $r+1$ almost everywhere.
\end{theorem}
\begin{proof}
The change of variables used in Theorem~\ref{thm:discrete} gives the multiplier and its optimal bounds. After a fixed deletion set $D$, the remaining multiplier is $w_D=\sum_{i\notin D}b_i$. The smallest possible value at each $z$ is obtained by deleting the $r$ largest terms. Since there are only finitely many deletion sets, the pointwise minimum and the essential infimum can be interchanged:
$\essinf_z\min_Dw_D(z)=\min_D\essinf_z w_D(z)$.
This proves \eqref{eq:robustbound}, including necessity. A finite indicator sum takes finitely many values, so strict positive coverage is equivalent to a positive essential lower bound. Equal weights reduce the formula to $v^2(\#\{i:z\in E_i\}-r)_+$.
\end{proof}

For a known surviving set, the canonical inverse is obtained by dividing its adjoint by $w_D$. In the above notation it is
$w_D(z)^{-1}\sum_{i\notin D}v_i\ind_{I_i}(z+t_i)y_i(z+t_i)$.

\begin{corollary}[Sharp uniform-grid erasure threshold]\label{cor:uniformerasure}
For \eqref{eq:discreteanalysis} and $0\le r\le N$, stable reconstruction survives every deletion of at most $r$ sample times if and only if
\begin{equation}\label{eq:erasurethreshold}
d\ge(r+1)\Delta,\qquad (N-r)\Delta\ge L.
\end{equation}
When this holds, the optimal common squared lower bound is
$A^{(r)}=A_{\Delta,N}-r\Delta$.
\end{corollary}
\begin{proof}
Theorems~\ref{thm:discrete} and \ref{thm:erasures} require
$\min\{\lfloor d/\Delta\rfloor,N-\lceil L/\Delta\rceil+1\}\ge r+1$.
These two integer inequalities are exactly \eqref{eq:erasurethreshold}; subtracting $r$ from the minimum gives the bound.
\end{proof}

\section{Stability, noise and finite resolution}\label{sec:regularization}
\subsection{Cutoff and projection reconstruction}
The inverse in Theorem~\ref{thm:exact} reconstructs a prescribed initial interval. For a state extending farther toward the horizon, a cutoff introduces a separate error. Define $X_-=L^2(( -\infty,0))$, let $Q_L$ be multiplication by $\ind_{[-L,0]}$, and let $\Oo_T:X_-\to Y_T$ have the same formula as \eqref{eq:Odef}. It is bounded, with squared norm $\min(d,T)$.

\begin{theorem}[Noise and finite-resolution error]\label{thm:regularization}
Fix $0<L<T$ and a rank-$K$ orthogonal projection $P_K$ on $X_L$. If
$y=\Oo_Tf+\eta$, $f\in X_-$ and $\|\eta\|_{Y_T}\le\delta$, define, with zero extension back to $X_-$,
\begin{equation}\label{eq:estimator}
\widehat f_{L,K}=P_K R_{T,L}y.
\end{equation}
Then
\begin{equation}\label{eq:error}
\|f-\widehat f_{L,K}\|^2
\le\|(I-Q_L)f\|^2+\|(I-P_K)Q_Lf\|^2
+\frac{\delta^2}{A(T,L)}.
\end{equation}
The analogous statement holds for discrete data on $X_-$ whenever \eqref{eq:coveragecondition} holds, replacing $R_{T,L}$ and $A(T,L)$ by their discrete counterparts.

For the auxiliary operator
\begin{equation}\label{eq:auxD}
D_L=-\partial_z^2,\qquad
\Dom(D_L)=H^2(-L,0)\cap H^1_0(-L,0),
\end{equation}
let $P_K$ project onto its first $K$ normalized sine eigenfunctions. If, for a stated $\nu>0$, $g=Q_Lf\in\Dom(D_L^{\nu/2})$, then
\begin{equation}\label{eq:sourceerror}
\|(I-P_K)g\|\le
\left(\frac{L}{\pi(K+1)}\right)^\nu\|D_L^{\nu/2}g\|.
\end{equation}
\end{theorem}
\begin{proof}
The characteristic integral gives $R_{T,L}\Oo_Tf=Q_Lf$. Initial coordinates outside $[-L,0]$ follow different characteristics and do not enter this identity. Therefore
$\widehat f_{L,K}=P_KQ_Lf+P_KR_{T,L}\eta$.
The three vectors $(I-Q_L)f$, $(I-P_K)Q_Lf$ and $P_KR_{T,L}\eta$ are mutually orthogonal in $X_-$. Pythagoras and the exact inverse norm prove \eqref{eq:error}. The discrete multiplier formula proves the same identity and argument for discrete data.

The eigenvalues in \eqref{eq:auxD} are $\lambda_n=(n\pi/L)^2$ and the eigenfunctions are $\sqrt{2/L}\sin(n\pi(z+L)/L)$. Writing $g=\sum_n c_ne_n$, the source hypothesis is exactly $\sum_n\lambda_n^\nu|c_n|^2<\infty$. Thus
$\sum_{n>K}|c_n|^2\le\lambda_{K+1}^{-\nu}\sum_{n>K}\lambda_n^\nu|c_n|^2$, proving \eqref{eq:sourceerror}.
\end{proof}

The three terms in \eqref{eq:error} are the discarded near-horizon tail, approximation error, and noise amplification. The first may include data that were recorded but deliberately omitted by the reconstruction cutoff. The Dirichlet endpoints in \eqref{eq:auxD} define only an approximation basis; they are not boundaries for the wave evolution. Its fractional-domain assumption is explicit because a sharp spatial cutoff need not preserve arbitrary Sobolev regularity. There is no universal relation between $K$ and $\varepsilon$.

\begin{corollary}[A conditional balance at critical time]\label{cor:balance}
Let $f\in X_T$ be observed by $\Oo_{T,T}$ with noise at most $\delta>0$. Assume constants $C_0,C_1,\beta>0$, $0<\nu<1/2$ and $0<\sigma_0<\min(d,T)$ such that, for every $0<\sigma\le\sigma_0$,
\begin{equation}\label{eq:criticalsource}
\|f\|_{L^2(-T,-T+\sigma)}\le C_0\sigma^\beta,
\qquad
\|D_{T-\sigma}^{\nu/2}Q_{T-\sigma}f\|\le C_1.
\end{equation}
Set
\begin{equation}\label{eq:sigmabalance}
\sigma_* =\left(\frac{\delta^2}{2\beta C_0^2}\right)^{1/(2\beta+1)}.
\end{equation}
If $\sigma_*\le\sigma_0$, choose $L=T-\sigma_*$ and an integer $K\ge1$ satisfying
\[
K+1\ge\frac T\pi(C_1/C_0)^{1/\nu}\sigma_*^{-\beta/\nu}.
\]
Then the estimator \eqref{eq:estimator} satisfies
\begin{equation}\label{eq:balancedrate}
\|f-\widehat f_{L,K}\|
\le\sqrt{2+2\beta}\,C_0\sigma_*^\beta
=O\big(\delta^{2\beta/(2\beta+1)}\big).
\end{equation}
The choice \eqref{eq:sigmabalance} minimizes the tail-plus-noise upper bound $C_0^2\sigma^{2\beta}+\delta^2/\sigma$; no minimax optimality is asserted.
\end{corollary}
\begin{proof}
Here $A(T,T-\sigma)=\sigma$ because $\sigma\le d$. Theorem~\ref{thm:regularization} bounds the squared error by
$C_0^2\sigma^{2\beta}+C_1^2(T/(\pi(K+1)))^{2\nu}+\delta^2/\sigma$.
Differentiation of the first and third terms gives \eqref{eq:sigmabalance} and
$\delta^2/\sigma_*=2\beta C_0^2\sigma_*^{2\beta}$. The choice of $K$ bounds the middle term by $C_0^2\sigma_*^{2\beta}$.
\end{proof}

The hypotheses \eqref{eq:criticalsource} describe a source class, not a consequence of the metric. The restriction $\nu<1/2$ avoids forcing an endpoint trace condition at every varying cutoff; the displayed fractional-domain bounds are still assumptions. The corresponding radial cutoff is $\varepsilon_{\sigma}=\rho_0e^{-\kappa(T-\sigma)}$. At fixed $T$ it approaches the critical positive radius, not the horizon, as $\sigma\downarrow0$.

\subsection{Continuous time and a spectral-band quadrature estimate}
The preceding finite-resolution spaces are not spectral subspaces of the full-line translation generator. For the latter kind of restriction, one can instead quantify time quadrature.

\begin{proposition}[Time sampling on a spectral band]\label{prop:quadrature}
Let $H$ be self-adjoint on $\Hh$, with spectral measure $E_H$, and define $U(t)=e^{-itH}$. Let $C:\Hh\to Y$ be bounded, $\Omega\ge0$, $T=N\Delta$ with $N\ge1$, and $f\in\ran E_H([-\Omega,\Omega])$. Then
\begin{equation}\label{eq:quadrature}
\left|\int_0^T\|CU(t)f\|^2dt
-\Delta\sum_{k=0}^{N-1}\|CU(k\Delta)f\|^2\right|
\le\|C\|^2\Omega T\Delta\,\|f\|^2.
\end{equation}
Consequently continuous squared bounds $a,b$ on a stated subspace of this spectral band transfer to bounds
$a-\|C\|^2\Omega T\Delta$ and $b+\|C\|^2\Omega T\Delta$ for the displayed discrete map, provided the lower one is positive.
\end{proposition}
\begin{proof}
Spectral calculus gives $\|(U(t)-U(s))f\|\le\Omega|t-s|\|f\|$ and $\|U(t)f\|=\|f\|$. Hence
$|\|CU(t)f\|^2-\|CU(s)f\|^2|\le2\|C\|^2\Omega|t-s|\|f\|^2$.
Integrating this estimate with $s=k\Delta$ over each cell and summing gives $\|C\|^2\Omega N\Delta^2\|f\|^2$, which is \eqref{eq:quadrature}.
\end{proof}

This proposition uses the left-endpoint grid $0,\ldots,N-1$, whereas Theorem~\ref{thm:discrete} includes the endpoint $N\Delta$. Adding that endpoint to \eqref{eq:quadrature} adds at most $\Delta\|C\|^2\|f\|^2$ to the error estimate. On the full line, finite spectral bandwidth generally gives an infinite-dimensional subspace. A nonzero compactly supported function cannot be exactly Fourier-bandlimited: its compact-support Fourier transform extends to an entire function, which cannot vanish on an open real interval without vanishing identically. Thus the spatial-support assumptions of Section~\ref{sec:exact} and exact band limitation are separate models. A finite-dimensional spectral calculation, including aliasing, is given in Appendix~\ref{app:galerkin}.

\section{A separate detector-attenuation model}\label{sec:gain}
Suppose that the recorded amplitude in \eqref{eq:Odef} is multiplied by a specified response $e^{-\gamma t}$, with $\gamma>0$. Let this bounded observation be denoted by $\Oo_{\gamma,T,L}$. The wave still evolves by $V(t)$.

\begin{proposition}[Gain-weighted Gramian]\label{prop:gain}
The Gramian is multiplication by
\begin{equation}\label{eq:gainweight}
w_{\gamma,T}(z)=e^{2\gamma z}
\frac{1-e^{-2\gamma\min(d,(T+z)_+)}}{2\gamma},\qquad -L\le z\le0.
\end{equation}
The optimal squared bounds are $w_{\gamma,T}(-L)$ and $w_{\gamma,T}(0)$. In particular, stable recovery holds exactly when $T>L$. Under the full-traversal condition $T\ge L+d$, put
$c_{\gamma,d}=(1-e^{-2\gamma d})/(2\gamma)$. Then
\begin{equation}\label{eq:gainbounds}
A_\gamma=c_{\gamma,d}e^{-2\gamma L},\qquad
B_\gamma=c_{\gamma,d},\qquad
\cond(\Oo_{\gamma,T,L})=e^{\gamma L}
=\left(\frac{\rho_0}{\varepsilon}\right)^{\gamma/\kappa}.
\end{equation}
For $T>L$, the canonical inverse is $M_{1/w_{\gamma,T}}\Oo_{\gamma,T,L}^*$ and has norm $A_\gamma^{-1/2}$, where for incomplete traversal $A_\gamma=w_{\gamma,T}(-L)$.
\end{proposition}
\begin{proof}
The change of variables in Theorem~\ref{thm:exact} now gives
$w_{\gamma,T}(z)=\int_{[0,T]\cap[-z,d-z]}e^{-2\gamma t}dt$, which evaluates to \eqref{eq:gainweight}. On a partially traversed characteristic it equals $(e^{2\gamma z}-e^{-2\gamma T})/(2\gamma)$; on a fully traversed one it equals $c_{\gamma,d}e^{2\gamma z}$. Both are increasing in $z$, and the pieces agree at their junction. Essential extrema therefore give the asserted optimal bounds. Full traversal yields \eqref{eq:gainbounds}, and Lemma~\ref{lem:inverse} gives the inverse.
\end{proof}

The exponent $\gamma/\kappa$ combines logarithmic optical travel time with an independently chosen detector law. There is no identification of a freely chosen power with surface gravity or a temperature. Even in the full-traversal regime, the growing condition number in \eqref{eq:gainbounds} describes this recording response, not dissipation by Rindler evolution.

\section{Schwarzschild near-horizon comparison}\label{sec:schwarzschild}
\subsection{Scalar reduction and the exponentially small potential}
Consider the Schwarzschild exterior with $M>0$,
\begin{equation}\label{eq:schwarzmetric}
g=-F(r)\,dt^2+F(r)^{-1}\,dr^2+r^2d\Omega^2,
\qquad F(r)=1-\frac{2M}{r},\quad r>2M.
\end{equation}
Here $t$ has the usual asymptotic Killing normalization. Fix an angular momentum $\ell\in\N_0$ and one normalized spherical harmonic $Y_{\ell j}$, with $-\Delta_{\mathbb S^2}Y_{\ell j}=\ell(\ell+1)Y_{\ell j}$. The index $j$ will not enter the radial estimates.

\begin{proposition}[Scalar radial operator]\label{prop:Schwarzoperator}
Under the substitution $\Phi(t,r,\omega)=r^{-1}\psi(t,x)Y_{\ell j}(\omega)$, where
\begin{equation}\label{eq:tortoise}
x=r_*=r+2M\log\left(\frac r{2M}-1\right),
\end{equation}
the equation $\Box_g\Phi=0$ becomes
\begin{equation}\label{eq:scalarSchwarz}
\psi_{tt}-\psi_{xx}+V_\ell(x)\psi=0,\qquad
V_\ell(x)=F(r(x))\left(\frac{\ell(\ell+1)}{r(x)^2}+\frac{2M}{r(x)^3}\right).
\end{equation}
On $L^2(\R,dx)$, $L_\ell=-\partial_x^2+V_\ell$ is nonnegative self-adjoint with domain $H^2(\R)$ and form domain $H^1(\R)$. Its square root and wave group are defined by the spectral construction of Proposition~\ref{prop:waveenergy}.

For $\kappa=1/(4M)$, define the explicit constants
\begin{equation}\label{eq:Cl}
C_\ell=\frac{\ell(\ell+1)+1}{4eM^2},\qquad
\overline v_\ell=\frac{\ell(\ell+1)+1}{4M^2}.
\end{equation}
Then, for all $x\in\R$,
\begin{equation}\label{eq:Vbounds}
0<V_\ell(x)\le\overline v_\ell,\qquad
V_\ell(x)\le C_\ell e^{2\kappa x}.
\end{equation}
Moreover $V_\ell(x)=C_\ell e^{2\kappa x}(1+O(e^{2\kappa x}))$ as $x\to-\infty$.
\end{proposition}
\begin{proof}
The scalar wave operator in \eqref{eq:schwarzmetric} is
\[
\Box_g\Phi=-F^{-1}\Phi_{tt}
+r^{-2}\partial_r(r^2F\partial_r\Phi)+r^{-2}\Delta_{\mathbb S^2}\Phi.
\]
After substitution and multiplication by $rF$, the radial derivative terms are $F^2\psi_{rr}+FF'\psi_r$, while the zeroth-order term is
$-F(F'/r+\ell(\ell+1)/r^2)\psi$. Since $dx/dr=F^{-1}$ and $F'=2M/r^2$, this proves \eqref{eq:scalarSchwarz}. The tortoise coordinate maps $(2M,\infty)$ onto $\R$. The real bounded potential is a bounded self-adjoint perturbation of the full-line Laplacian, so its domain is $H^2$; the form proves nonnegativity and gives domain $H^1$. Both infinite ends remain part of the full-line problem, with no reflecting cutoff.

The first bound in \eqref{eq:Vbounds} follows from $F<1$ and $r>2M$. For the sharper horizon estimate set $y=r/(2M)-1>0$. Then
\[
2\kappa x=1+y+\log y,\qquad
V_\ell(x)=\frac{y}{4M^2}
\left(\frac{\ell(\ell+1)}{(1+y)^3}+\frac1{(1+y)^4}\right).
\]
The relation $y=e^{2\kappa x-1-y}\le e^{2\kappa x-1}$ gives the second bound. As $x\to-\infty$, it also gives
$y=e^{2\kappa x-1}(1+O(e^{2\kappa x}))$; substitution yields the asymptotic with the displayed leading constant.
\end{proof}

This is the scalar effective potential, not the gravitational Regge--Wheeler potential. The scalar energy and global black-hole wave problem are discussed in \cite{DR}; the derivation above fixes the convention needed here.

To translate the apparatus, set $s=x+a$, with $a\ge0$, and
\begin{equation}\label{eq:translatedV}
V_a(s)=V_\ell(s-a),\qquad L_a=-\partial_s^2+V_a(s).
\end{equation}
The initial characteristic interval remains $[-L,0]$ and the detector remains $[0,d]$ in $s$. In $x$ both move toward $-\infty$. Although $\|V_a\|_\infty=\|V_\ell\|_\infty$ does not tend to zero, on the causal half-line relevant up to time $T$ we have
\begin{equation}\label{eq:deltaa}
\delta_a:=\sup_{s\le T}V_a(s)
\le C_\ell e^{2\kappa T}e^{-2\kappa a}.
\end{equation}
The distinction between local smallness and a globally small bounded perturbation is essential.

\subsection{Anchored finite-energy data without a zero-integral condition}
Let $\Ee_a$ be the completion of $C_c^\infty(\R)\times C_c^\infty(\R)$ in the norm
\begin{equation}\label{eq:energyspacea}
\|(u,v)\|_{\Ee_a}^2
=\int_\R\big(|v|^2+|u'|^2+V_a|u|^2\big)\,ds.
\end{equation}
This is the conserved Killing-energy norm for the reduced equation, up to the factor $1/2$. In particular it is not an assertion about a nondegenerate energy norm across the horizon.

To check the physical normalization, the scalar Killing energy of one angular mode, after the substitution $\Phi=r^{-1}\psi Y_{\ell j}$, is
\[
\frac12\int_\R\left(|\psi_t|^2+
|\psi_x-(F/r)\psi|^2+
\frac{F\ell(\ell+1)}{r^2}|\psi|^2\right)dx.
\]
For compact data, integration of the cross term uses
$\partial_x(F/r)+(F/r)^2=FF'/r$ and gives one half of \eqref{eq:energyspacea}. Equality extends to the completed energy space by density. For the anchored data below, the boundary term $(F/r)|\psi|^2$ vanishes at the horizon end and on the outer side.

\begin{lemma}[Energy completion and finite propagation]\label{lem:energycompletion}
The wave group of $L_a$ extends uniquely to a strongly continuous unitary group on $\Ee_a$. Its displacement has a unique locally $H^1$ representative, its velocity belongs to $L^2$, and it solves
$u_{tt}-u_{ss}+V_au=0$ in distributions. If initial displacement and velocity vanish for $s>b$, the solution vanishes for $s>b+t$ for $t\ge0$.

For every $f\in X_L$, define
\begin{equation}\label{eq:anchoreddata}
u_f(s)=\begin{cases}\displaystyle\int_s^0 f(z)\,dz,&s\le0,\\0,&s>0,\end{cases}
\qquad v_f=f,
\end{equation}
where $f$ is zero outside $[-L,0]$. This pair belongs to $\Ee_a$, satisfies $u_f'=-f$, and has $q_+(0)=f$, $q_-(0)=0$. With
\begin{equation}\label{eq:ca}
b_L=\frac{1-e^{-2\kappa L}}{4\kappa^2},\qquad
c_a=\left(2+C_\ell b_L e^{-2\kappa a}\right)^{1/2},
\end{equation}
its energy obeys
\begin{equation}\label{eq:initialenergy}
2\|f\|^2\le\|(u_f,f)\|_{\Ee_a}^2\le c_a^2\|f\|^2.
\end{equation}
\end{lemma}
\begin{proof}
On every compact interval $K$, the potential has a positive minimum, so
$\|u\|_{H^1(K)}^2\le (1+1/\min_K V_a)\int(|u'|^2+V_a|u|^2)$.
Consequently energy-Cauchy displacements have consistent locally $H^1$ limits; their derivatives and their products with $\sqrt{V_a}$ converge globally in $L^2$. This also proves uniqueness of the representative. The spectral rotation \eqref{eq:rotation}, with $H_a=L_a^{1/2}$, defines the unitary extension. Its energy-coordinate generator has domain $\Dom(H_a)\oplus\Dom(H_a)=H^1\oplus H^1$ and is the skew-adjoint matrix stated in Proposition~\ref{prop:waveenergy}. Test-function limits prove the distributional equation. Thus no expression involving $H_a u$ as an $L^2$-operator on a non-$L^2$ displacement is needed.

For smooth compact data, let $e=(|u_t|^2+|u_s|^2+V_a|u|^2)/2$. The local identity is
$\partial_t e=\partial_s\operatorname{Re}(u_t\overline{u_s})$.
Differentiating the energy on $(b+t,\infty)$ gives
\[
\frac d{dt}\int_{b+t}^\infty e\,ds
=-\frac12|u_t+u_s|^2(t,b+t)-\frac12V_a(b+t)|u(t,b+t)|^2\le0.
\]
Zero initial exterior energy therefore remains zero. Energy approximation extends this conclusion to completed data supported on the same half-line; the local $H^1$ convergence preserves the zero representative there. For completeness, an energy displacement with $u'\in L^2$ satisfies $u(s)=o(\sqrt{|s|})$ as $s\to-\infty$: split its primitive at a fixed negative point and make the derivative norm in the remaining tail arbitrarily small. A cutoff between $-2R$ and $-R$ therefore has vanishing added derivative energy, while its potential-energy error tends to zero by integrability. After this cutoff, ordinary $H^1$ approximation with zero trace at $b$ supplies smooth data supported to the left of $b$. Velocities are approximated in $L^2$ on the same half-line.

For \eqref{eq:anchoreddata}, $u_f$ is continuous, locally $H^1$, constant on $(-\infty,-L)$, and zero on $(0,\infty)$. Cauchy--Schwarz gives
$|u_f(s)|^2\le\min(L,-s)\|f\|^2$ for $s\le0$. Thus
\[
\int V_a|u_f|^2\,ds
\le C_\ell e^{-2\kappa a}\|f\|^2
\int_{-\infty}^0e^{2\kappa s}\min(L,-s)\,ds
=C_\ell b_L e^{-2\kappa a}\|f\|^2.
\]
To verify membership in the completion, cut the constant left tail to zero between $-2R$ and $-R$, for $R>L$, with derivative bounded by a constant times $R^{-1}$. The added squared derivative norm is $O(R^{-1}|\int f|^2)$, and the potential tail tends to zero by \eqref{eq:Vbounds}. The resulting compactly supported $H^1$ functions can be smoothed in the form norm, while $f$ is approximated in $L^2$. This proves membership. Finally $u_f'=-f$ gives \eqref{eq:initialenergy} and the characteristic data.
\end{proof}

If $\int f\ne0$, then $u_f\notin L^2(\R)$, but its potential energy is finite. Requiring an $L^2$ displacement compactly supported in $[-L,0]$ would unnecessarily impose $\int f=0$. The anchoring in \eqref{eq:anchoreddata} is part of the initial-data model: it excludes an independent constant displacement parameter. The theorem below reconstructs this family of prepared energy data, not arbitrary pairs in $\Ee_a$.

\subsection{Observation-operator comparison and reconstruction}
Let $u_a$ be the energy solution with data \eqref{eq:anchoreddata}, and define
\begin{equation}\label{eq:Oa}
q_a(t,s)=\frac{u_{a,t}-u_{a,s}}2,
\qquad (\Oo_af)(t,s)=q_a(t,s)\big|_{(0,T)\times[0,d]}.
\end{equation}
This is a bounded map $X_L\to Y_T$, since energy conservation bounds $q_a$ in $C([0,T];L^2)$. The potential can generate an incoming component for $t>0$; only the initial data are assumed to be purely outgoing.
Define
\begin{align}
F_\kappa(T)&=\left(\int_0^T(e^{2\kappa t}-1)^2dt\right)^{1/2}
=\left(\frac{e^{4\kappa T}-1}{4\kappa}
-\frac{e^{2\kappa T}-1}{\kappa}+T\right)^{1/2},\label{eq:Fkappa}\\
D_a&=\frac{c_a C_\ell}{16\kappa^2}e^{-2\kappa a},\qquad
h_a=D_a F_\kappa(T).\label{eq:ha}
\end{align}

\begin{theorem}[Finite-energy Schwarzschild comparison]\label{thm:Schwarzcomparison}
Fix $M>0$, $\ell\in\N_0$, and $L,d,T>0$. For every $a\ge0$ and every $f\in X_L$, with no condition on its integral,
\begin{equation}\label{eq:pointcomparison}
\|q_a(t)-V(t)f\|_{L^2(\R)}
\le D_a(e^{2\kappa t}-1)\|f\|,
\qquad 0\le t\le T.
\end{equation}
Consequently
\begin{equation}\label{eq:Ocomparison}
\|\Oo_a-\Oo_{T,L}\|_{\mathcal B(X_L,Y_T)}\le h_a
=O(e^{-2\kappa a}).
\end{equation}
If $T>L$ and $h_a<\sqrt{A(T,L)}$, the optimal squared bounds $A_a,B_a$ of $\Oo_a$ satisfy
\begin{equation}\label{eq:Schwarzbounds}
A_a\ge(\sqrt{A(T,L)}-h_a)^2>0,
\qquad B_a\le(\sqrt{B(T,L)}+h_a)^2.
\end{equation}
The canonical inverse $R_a=(\Oo_a^*\Oo_a)^{-1}\Oo_a^*$ reconstructs exactly and stably, with
\begin{equation}\label{eq:Schwarzinversebound}
\|R_a\|\le\frac1{\sqrt{A(T,L)}-h_a},\qquad
\cond(\Oo_a)\le
\frac{\sqrt{B(T,L)}+h_a}{\sqrt{A(T,L)}-h_a}.
\end{equation}
The constants in \eqref{eq:ha} and \eqref{eq:Schwarzbounds} are explicit sufficient constants, not asserted to be optimal.
\end{theorem}
\begin{proof}
Finite propagation gives $u_a(t,s)=0$ for $s>t$. Since the displacement is locally $H^1$ and its derivative is square integrable, anchoring at this zero region gives, for $s\le t$,
\begin{equation}\label{eq:anchorpointwise}
|u_a(t,s)|^2=\left|\int_s^t u_{a,s}(t,z)\,dz\right|^2
\le(t-s)\|u_{a,s}(t)\|_2^2
\le(t-s)c_a^2\|f\|^2.
\end{equation}
The exponential potential bound therefore yields
\begin{align}
\|V_au_a(t)\|_2^2
&\le C_\ell^2 e^{-4\kappa a}c_a^2\|f\|^2
\int_{-\infty}^t e^{4\kappa s}(t-s)\,ds\notag\\
&=\frac{C_\ell^2c_a^2}{16\kappa^2}
e^{-4\kappa a}e^{4\kappa t}\|f\|^2.\label{eq:forcingbound}
\end{align}
The characteristic equation, obtained directly from \eqref{eq:scalarSchwarz}, is
$(\partial_t+\partial_s)q_a=-V_au_a/2$, with initial value $f$.
Here $V_au_a\in C([0,T];L^2)$: multiplication by $V_a$ is bounded from the displacement energy norm into $L^2$, because $\|V_au\|_2\le\|V_a\|_\infty^{1/2}\|\sqrt{V_a}u\|_2$. Thus the distributional characteristic equation has the $L^2$ Duhamel formula
\begin{equation}\label{eq:qDuhamel}
q_a(t)=V(t)f-\frac12\int_0^tV(t-r)(V_au_a(r))\,dr.
\end{equation}
It follows either by differentiating for smooth data and passing to energy limits, or by uniqueness of the forced translation equation with an $L^1_tL^2_s$ source. Unitarity and \eqref{eq:forcingbound} give
\[
\|q_a(t)-V(t)f\|_2
\le\frac{c_aC_\ell e^{-2\kappa a}}{8\kappa}
\int_0^t e^{2\kappa r}dr
=D_a(e^{2\kappa t}-1)\|f\|.
\]
This is \eqref{eq:pointcomparison}. Restriction to the shell and integration in time prove \eqref{eq:Ocomparison}. Since $c_a\le c_0$ for $a\ge0$, the stated asymptotic follows with all other parameters fixed. Lemma~\ref{lem:inverse} and Theorem~\ref{thm:exact} now prove \eqref{eq:Schwarzbounds} and \eqref{eq:Schwarzinversebound}.
\end{proof}

For example, a completely explicit sufficient condition on the apparatus translation is
\begin{equation}\label{eq:explicita}
e^{-2\kappa a}
<\frac{16\kappa^2\sqrt{A(T,L)}}{c_0 C_\ell F_\kappa(T)},
\qquad T>L.
\end{equation}
The result is also an observation inequality on the prepared energy-data subspace: the lower bound relative to $\|(u_f,f)\|_{\Ee_a}^2$ is
$(\sqrt A-h_a)^2/c_a^2$, and an upper bound is $(\sqrt B+h_a)^2/2$, by \eqref{eq:initialenergy}.

\begin{corollary}[A constructive left inverse and noisy data]\label{cor:Schwarzreconstruct}
Under the positive-bound hypotheses of Theorem~\ref{thm:Schwarzcomparison}, put
$E_a=\Oo_a-\Oo_{T,L}$ and $R_0=R_{T,L}$. Then
\begin{equation}\label{eq:Neumanninverse}
\mathcal R_a=\sum_{n=0}^{\infty}(-R_0E_a)^nR_0
\end{equation}
converges in operator norm, is a bounded left inverse of $\Oo_a$, and has norm at most $(\sqrt A-h_a)^{-1}$. It need not be the canonical inverse. For $y=\Oo_af+\eta$ with $\|\eta\|\le\delta$, either inverse has reconstruction error at most $\delta/(\sqrt A-h_a)$. For a rank-$K$ orthogonal projection $P_K$ on $X_L$,
\begin{equation}\label{eq:Schwarzfinite}
\|f-P_KR_ay\|^2
\le\|(I-P_K)f\|^2+\frac{\delta^2}{(\sqrt A-h_a)^2}.
\end{equation}
\end{corollary}
\begin{proof}
We have $R_0\Oo_a=I+R_0E_a$ and $\|R_0E_a\|\le h_a/\sqrt A<1$. The Neumann series proves the inverse identity and gives
$\|\mathcal R_a\|\le A^{-1/2}/(1-h_a/\sqrt A)$. The noise bound follows by applying an inverse to $\eta$. In \eqref{eq:Schwarzfinite}, the approximation error is orthogonal to the projected noise, so Pythagoras applies.
\end{proof}

The exact tail separation in Theorem~\ref{thm:regularization} relied on free characteristics. It is not asserted here for unknown Schwarzschild data outside the prepared interval: potential scattering can couple such data into the recording region.

\begin{corollary}[Discrete Schwarzschild observations]\label{cor:Schwarzdiscrete}
Let $T=N\Delta$ and define
$\Aa_a f=(\sqrt\Delta\,q_a(k\Delta)|_I)_{k=0}^N$. Then
\begin{equation}\label{eq:discreteha}
\|\Aa_a-\Aa_{\Delta,N}\|\le
h_{a,\Delta,N}:=D_a\left(\Delta\sum_{k=0}^N
(e^{2\kappa k\Delta}-1)^2\right)^{1/2}.
\end{equation}
If \eqref{eq:coveragecondition} holds and $h_{a,\Delta,N}<\sqrt{A_{\Delta,N}}$, valid squared bounds are
$(\sqrt{A_{\Delta,N}}-h_{a,\Delta,N})^2$ and
$(\sqrt{B_{\Delta,N}}+h_{a,\Delta,N})^2$.
The same conclusion holds for any surviving set of samples after erasures, using its free coverage bounds and the sum over that surviving set in \eqref{eq:discreteha}.
\end{corollary}
\begin{proof}
Apply \eqref{eq:pointcomparison} at each sample time, square and sum with weights $\Delta$. Lemma~\ref{lem:inverse} gives the transferred bounds. Restriction to a surviving set gives the last assertion.
\end{proof}

The free problem has a literal spatial fusion-frame structure because its Gramian is a coverage multiplier. The perturbed Schwarzschild map inherits a frame inequality and a bounded inverse; its Gramian need not remain a multiplier or the sum of the same spatial projections.

\subsection{The travel-time obstruction persists for Schwarzschild}
The preceding comparison gives a sufficient supercritical observation condition near the horizon. Finite propagation also gives a necessary time condition at every apparatus position, including the critical case where a simple support argument alone would be insufficient.

\begin{proposition}[Subcritical and critical obstruction]\label{prop:Schwarzcritical}
For the data and detector in \eqref{eq:anchoreddata}--\eqref{eq:Oa}, fix any $a\ge0$.
If $T<L$, then $L^2((-L,-T))\subset\ker\Oo_a$.
If $T=L$, the map $\Oo_a$ has no positive lower bound on $X_L$. In particular, $T>L$ is necessary for stable recovery for every $a$.
\end{proposition}
\begin{proof}
If $f$ is supported in $(-L,-T)$, its anchored displacement vanishes for $s>-T$. Finite propagation gives $u_a(t,s)=0$ for $s>-T+t$, so the shell records zero for $0<t<T$.

At $T=L$, let $0<\eta<L$ and choose a smooth, mean-zero $f_\eta$ supported in $(-L,-L+\eta)$ with $\|f_\eta\|=1$. Its anchored primitive $u_\eta$ is supported in the same interval and satisfies the Dirichlet Poincar\'e estimate $\|u_\eta\|_2\le\eta/\pi$. The free displacement is $u^{(0)}(t,s)=u_\eta(s-t)$, since its initial velocity is $-u_\eta'=f_\eta$.

For these square-integrable data the free-wave Duhamel formula, using
$\|\sin(t|D|)/|D|\|\le t$, gives
\[
\|u_a(t)\|_2\le\|u_\eta\|_2+
\overline v_\ell\int_0^t(t-r)\|u_a(r)\|_2\,dr.
\]
Iteration of this Volterra inequality yields
$\|u_a(t)\|_2\le\|u_\eta\|_2\cosh(\sqrt{\overline v_\ell}\,t)$.
Indeed, the $n$-th iterated constant term is
$\|u_\eta\|_2\overline v_\ell^n t^{2n}/(2n)!$, and the remainder tends to zero because $\|u_a(t)\|_2$ is bounded on a compact time interval. Applying \eqref{eq:qDuhamel} therefore gives
\[
\|q_a(t)-V(t)f_\eta\|_2
\le\frac{\sqrt{\overline v_\ell}}2
\sinh(\sqrt{\overline v_\ell}\,t)\|u_\eta\|_2.
\]
At critical time Theorem~\ref{thm:exact} gives $\|\Oo_{L,L}f_\eta\|\le\sqrt\eta$. Consequently
\begin{equation}\label{eq:criticalpacket}
\|\Oo_af_\eta\|
\le\sqrt\eta+
\frac{\eta\sqrt{T\overline v_\ell}}{2\pi}
\sinh(\sqrt{\overline v_\ell}\,T)\longrightarrow0.
\end{equation}
The initial profiles have unit norm, proving failure of a lower bound. The mean-zero packets are used only as test vectors for this obstruction, not as a restriction in Theorem~\ref{thm:Schwarzcomparison}.
\end{proof}

We do not assert injectivity or characterize the range of the Schwarzschild map at critical time. The injective/nonclosed-range conclusion in Theorem~\ref{thm:exact} uses its exact free multiplier. Proposition~\ref{prop:Schwarzcritical} proves the stability obstruction without requiring that stronger classification.

\subsection{Geometric meaning of the comparison}
For completeness, the Rindler normalization can be read directly from the metric. The proper distance from the horizon on a static slice is
\[
\rho(r)=\int_{2M}^rF(\xi)^{-1/2}\,d\xi
=2\sqrt{2M(r-2M)}\big(1+O((r-2M)/M)\big).
\]
Integration of the expansion of $F^{-1/2}$ proves this formula. Hence
$F=(\kappa\rho)^2(1+O(\rho^2/M^2))$ and
\begin{equation}\label{eq:properoptical}
x=4M\log(\rho/(4M))+2M+O(\rho^2/M).
\end{equation}
The latter follows by substituting $r-2M=\rho^2/(8M)(1+O(\rho^2/M^2))$ into \eqref{eq:tortoise}. Thus, to leading order, $x$ is the Rindler optical coordinate with $\kappa=1/(4M)$. On a bounded translated $s$-interval the proper radii scale as $e^{-\kappa a}$, whereas the scalar potential scales as $e^{-2\kappa a}$.

Theorem~\ref{thm:Schwarzcomparison} is the corresponding statement for the specified observation operators, not just for metric coefficients. Its proof controls the noncompact left displacement tail and uses the exact exterior evolution. It requires fixed $M,\ell,L,d,T$. The constants grow with angular momentum and recording duration; no estimate uniform in all angular modes, or in a simultaneous long-time limit with a distant fixed detector, is claimed.

\section{Discussion and limitations}\label{sec:discussion}
The observation constants depend on the state preparation, detector geometry, time interval and data norm. Their dependence on $\kappa$ enters through the relation between radial and optical distance. In the massless $1+1$ Rindler problem, the optical equation itself is free transport for each characteristic component. The sharp Gramian is a specialization of characteristic observability in these coordinates. Rindler spacetime is flat; the horizon is represented by an infinite optical end in this coordinate description.

The Schwarzschild comparison concerns the actual curved-background dynamics. Anchored finite-energy preparation allows every outgoing profile in $X_L$, including profiles with nonzero integral. The scalar potential is integrable at the horizon end, so a constant displacement tail there has finite energy. The reconstructed family has zero incoming initial component and a fixed additive displacement; arbitrary Schwarzschild Cauchy pairs fall outside this preparation.

The reduced Killing-energy norm and the normalized characteristic-derivative observation specify the energy and measurement used in this model. They do not describe every physical sensor or establish regularity across the horizon in a nondegenerate energy norm. The discrete data consist of whole spatial profiles, so the positive discrete theorem concerns a finite collection of infinite-dimensional subspace measurements. This is consistent with the finite scalar-channel obstruction. Robustness under erased profiles is exactly the persistence of weighted spatial coverage.

At critical time the Rindler observation is injective but unstable, whereas the Schwarzschild result established here is only the failure of a positive lower bound. Above critical time, the comparison guarantees stability sufficiently near the horizon, with an explicit sufficient translation. It does not characterize the optimal Schwarzschild constants at an arbitrary location. Nor does it imply a full-line operator-norm limit of the translated propagators: the global potential norm remains unchanged.

Throughout the comparison, $M,\ell,L,d,T$ are fixed and the initial region and detector are translated together. Angular-mode summation, trapping, long-time observation at a fixed exterior detector, quasinormal-mode analysis, Kerr dynamics and quantized fields lie outside its scope. The classical reconstruction results imply neither Hawking radiation nor a statement about black-hole information. A fixed exterior detector with initial data approaching the horizon would require control on growing optical regions and increasing recording intervals.

\section{Conclusion}\label{sec:conclusion}
For the specified massless outgoing Rindler data, distributed shell recording has an exactly computable Gramian, a sharp optical travel-time threshold and an explicit canonical inverse. Uniform discrete recording is equivalent to a spatial fusion frame, with optimal bounds and sharp erasure thresholds. Finite-resolution reconstruction separates the discarded tail, approximation error and noise amplification; spectral-band time quadrature is controlled in a separate regime. A prescribed detector gain produces its own condition-number law without changing conservative propagation.

At fixed angular momentum, the actual Schwarzschild scalar equation admits a quantitative comparison on anchored finite-energy data without a zero-integral restriction. For fixed $M,\ell,L,d,T$, the observation operators differ by an explicit $O(e^{-2\kappa a})$ term. Thus, when $T>L$, sufficiently large translations toward the horizon admit stable reconstruction and a convergent perturbative left inverse. At or below the optical travel time, a positive lower bound is impossible for every apparatus translation.

\appendix
\section{Weighted multiplication and cutoff inversion}\label{app:weighted}
The following elementary comparison isolates a separate weighted inverse problem. It is not used to generate the wave dynamics.

\begin{proposition}\label{prop:weighted}
Let $\Hh=L^2(0,R)$, let $U:\Hh\to\ell^2$ be a unitary orthonormal-basis analysis, and let $w\in L^\infty(0,R)$ be positive almost everywhere. Then $C_w=UM_w$ is bounded on all of $\Hh$, with
\begin{equation}\label{eq:multiplieridentities}
\|C_w\|=\|w\|_\infty,\qquad C_w^*C_w=M_{w^2}.
\end{equation}
It is injective with dense range, and that range is closed if and only if $\essinf w>0$. In particular, for $w(x)=x^\alpha$, $\alpha>0$, its optimal squared bounds are $0$ and $R^{2\alpha}$, and its range is nonclosed.

Let $Q_\varepsilon=M_{\ind_{[\varepsilon,R]}}$ and
$m_\varepsilon=\essinf_{[\varepsilon,R]}w>0$. Let $P_K$ be an orthogonal projection onto $K$ coefficient coordinates, and write $\Pi_K=U^*P_KU$. The bounded regularized inverse
\[
R_{\varepsilon,K}=M_{1/w}Q_\varepsilon U^*P_K
\]
has norm at most $m_\varepsilon^{-1}$. For $y=C_wf+\eta$, $\|\eta\|\le\delta$,
\begin{equation}\label{eq:weightederror}
\|f-R_{\varepsilon,K}y\|
\le\|(I-Q_\varepsilon)f\|
+m_\varepsilon^{-1}\|(I-\Pi_K)(wf)\|
+m_\varepsilon^{-1}\delta.
\end{equation}
For another bounded weight $\widetilde w$,
$\|C_{\widetilde w}-C_w\|=\|\widetilde w-w\|_\infty$.
\end{proposition}
\begin{proof}
Unitarity of $U$ proves the identities, optimal bounds and perturbation norm. Positivity almost everywhere gives injectivity of $M_w$ and zero kernel of its adjoint, hence density of its range. A positive essential infimum gives a bounded inverse; if the essential infimum is zero, normalized vectors supported where $w$ is small show that the map is not bounded below. The closed-range characterization for an injective bounded operator gives the converse. In the cutoff inverse, $M_{1/w}Q_\varepsilon$ denotes the bounded multiplier $\ind_{[\varepsilon,R]}/w$; no unbounded full inverse is being applied. Expanding $R_{\varepsilon,K}C_wf$ and using $M_{1/w}Q_\varepsilon(wf)=Q_\varepsilon f$ proves \eqref{eq:weightederror} by the triangle inequality.
\end{proof}

For a vanishing weight, algebraic inversion on the range is possible but unbounded in the inherited data norm. The cutoff operator above is a regularizer, not automatically the Moore--Penrose inverse of the unrestricted analysis. The multiplication identities explain its stability cost without supplying a geometric or thermodynamic interpretation.

\section{Finite-dimensional spectral observations and aliasing}\label{app:galerkin}
\begin{proposition}\label{prop:aliasing}
Let $H_K$ be a self-adjoint operator on a $K$-dimensional Hilbert space, with orthonormal eigenvectors $e_r$ and real eigenvalues $\omega_r$. Let $C:\C^K\to\C^J$ be bounded and put $c_r=Ce_r$. For the observation $f\mapsto(\sqrt\Delta\,Ce^{-ik\Delta H_K}f)_{k=0}^N$, the Gram matrix in column coordinates is
\begin{equation}\label{eq:finiteGram}
G_{rs}=\Delta\,c_r^*c_s\sum_{k=0}^N e^{ik\Delta(\omega_r-\omega_s)}.
\end{equation}
Its smallest and largest eigenvalues are the optimal squared observation bounds. Let $p$ be the number of distinct phases $z=e^{-i\Delta\omega_r}$ and $P_z$ the orthogonal projections onto their phase eigenspaces. If $N+1\ge p$, the observation is injective, equivalently bounded below, if and only if $C$ is injective on every $\ran P_z$.
\end{proposition}
\begin{proof}
Expansion of the squared data norm gives \eqref{eq:finiteGram}. Its Rayleigh quotients give the bounds. If all observations vanish, then
$\sum_z z^k CP_zf=0$ for $k=0,\ldots,N$. The first $p$ equations form an invertible Vandermonde system in the $Y$-valued unknowns $CP_zf$, so each is zero. Injectivity on each phase space then gives $f=0$. Conversely, a nonzero vector in the kernel of one restricted $C$ is invisible at every time. In finite dimension injectivity is equivalent to a positive lower singular value.
\end{proof}

Aliasing occurs when $\Delta(\omega_r-\omega_s)\in2\pi\mathbb Z$. Multiple channels must distinguish the resulting phase eigenspaces. This is a finite-dimensional model with its own self-adjoint generator. Applying it as a Galerkin approximation to a PDE additionally requires an estimate for the error in the evolution and observation; such an estimate is not supplied by a mere spatial projection bound.

\section*{Data availability}
All results are analytical. No external datasets were used, and no new datasets were generated or analyzed.

\end{document}